\documentclass[11pt]{article}

\usepackage[a4paper,margin=1.1in]{geometry}
\usepackage[T1]{fontenc}
\usepackage[utf8]{inputenc}
\usepackage{lmodern}
\usepackage{amsmath,amssymb,amsthm,mathtools}
\usepackage{booktabs,longtable,array,enumitem}
\usepackage{microtype}
\usepackage{url}
\usepackage{hyperref}
\usepackage[nameinlink,noabbrev]{cleveref}
\hypersetup{colorlinks=true,linkcolor=blue,citecolor=blue,urlcolor=blue}

\newtheorem{theorem}{Theorem}[section]
\newtheorem{lemma}[theorem]{Lemma}

\newtheorem{construction}[theorem]{Construction}
\theoremstyle{definition}

\crefname{theorem}{Theorem}{Theorems}
\crefname{lemma}{Lemma}{Lemmas}
\crefname{proposition}{Proposition}{Propositions}
\crefname{corollary}{Corollary}{Corollaries}
\crefname{construction}{Construction}{Constructions}
\crefname{definition}{Definition}{Definitions}
\crefname{remark}{Remark}{Remarks}
\crefname{table}{Table}{Tables}
\crefname{section}{Section}{Sections}
\crefname{equation}{Equation}{Equations}

\newcommand{\Cxi}{C^{\xi}}

\newcommand{\calB}{\mathcal B}
\newcommand{\ex}{\Xi}
\newcommand{\ZZ}{\mathbb Z}

\title{Minimum-excess coverings of complete graphs by 3-, 4-, and 5-cliques}
\author{Petr Kov\'a\v{r}\thanks{Department of Applied Mathematics, VSB--Technical University of Ostrava, 17. listopadu 2172/15, 708 00 Ostrava, Czech Republic. Email: \texttt{petr.kovar@vsb.cz}.}
\and Yifan Zhang\thanks{Department of Applied Mathematics, VSB--Technical University of Ostrava, 17. listopadu 2172/15, 708 00 Ostrava, Czech Republic; Department of Mathematics, University of Ostrava, Ml\'ynsk\'a 702/5, 702 00 Ostrava, Czech Republic; Department of Algebra, Charles University, Sokolovsk\'a 49/83, 186 75 Prague, Czech Republic. Email: \texttt{yifan.zhang@vsb.cz}.}}
\date{}

\begin{document}
\maketitle

\begin{abstract}
We study two-level optimal coverings of the complete graph $K_v$ by cliques of orders $3$, $4$, and $5$.  The first level minimizes the excess, namely the number of repeated edge occurrences, and the second level minimizes the number of cliques among coverings with minimum excess.  We first isolate the two-size problem for triangles and 4-cliques, where the main congruence and local-degree methods already appear and where nonzero excess is unavoidable for one residue class of $v$.  This motivates the passage to quintuples.  For $\{K_3,K_4,K_5\}$ coverings we determine the minimum excess for every $v$, use an edge-count reduction for the secondary optimization, obtain exact values in ten residue classes modulo $20$, and give bounds for the remaining classes.
\end{abstract}



\section{Introduction}

Covering complete graphs by complete subgraphs is a classical problem in design theory.  A $t$-$(v,k,\lambda)$ covering is a pair $(X,\calB)$, where $|X|=v$ and $\calB$ is a collection of $k$-subsets of $X$ such that each $t$-subset of $X$ is contained in at least $\lambda$ blocks.  We work with $t=2$ and $\lambda=1$, so the objects are coverings of $K_v$ by complete subgraphs.

The motivating application is load distribution in parallel boundary element methods.  One wants to split a dense interaction pattern among processors while keeping repeated communication small; cyclic and clique decompositions of complete graphs naturally encode such splittings \cite{Lukas__2015,Dohr_2019}.  If all blocks have one size, minimizing the number of blocks is essentially tied to minimizing repeated pairs.  With several block sizes this is no longer true: larger cliques reduce the number of blocks but may increase repeated edges.  We therefore use a two-level optimum: first minimize the excess, and only then minimize the number of blocks among coverings with minimum excess.

The first mixed-clique case considered here is the two-size problem with $K_3$ and $K_4$.  It is simpler than the three-size problem, but it already contains the main ingredients used later: edge-count identities, local congruences from vertex degrees, truncations of group divisible designs, and small exceptional cases.  It also explains why adding quintuples is natural.  For $\{K_3,K_4\}$ coverings, a congruence obstruction forces nonzero excess whenever $v\equiv2\pmod3$; allowing $K_5$ removes this obstruction for all orders except $6$ and $8$.  The paper therefore treats the $\{K_3,K_4\}$ problem as the foundational two-size case and then develops the $\{K_3,K_4,K_5\}$ theory around a linear edge-count identity for the secondary optimization.

\section{Preliminaries and counting identities}

Let $S=\{k_1<\cdots<k_r\}$ be a set of allowed block sizes.  A $2$-$(v,S,1)$ covering is a pair $(X,\calB)$, where $|X|=v$, every block has size in $S$, and every pair of points of $X$ is contained in at least one block.  Repetition of blocks is permitted.

The \emph{excess multigraph} $\ex$ has an edge $xy$ with multiplicity $m(xy)-1$ whenever $xy$ is covered $m(xy)\ge 2$ times; isolated vertices are omitted.  Thus $|E(\ex)|$ is exactly the number of repeated edge occurrences.  Let $\xi_S(v)$ be the minimum possible value of $|E(\ex)|$.  Let
\[
        \Cxi(v,S,2)
\]
be the minimum number of blocks in a $2$-$(v,S,1)$ covering with excess size $\xi_S(v)$.

For $S=\{3,4,5\}$, write $\alpha,\beta,\gamma$ for the numbers of blocks of size $3,4,5$, respectively.  For a vertex $x$, write $\alpha_x,\beta_x,\gamma_x$ for the corresponding numbers of blocks containing $x$.  If $d_\ex(x)$ is the degree of $x$ in the excess multigraph, then
\begin{equation}\label{eq:local}
        2\alpha_x+3\beta_x+4\gamma_x=v-1+d_\ex(x).
\end{equation}
For a decomposition this becomes $2\alpha_x+3\beta_x+4\gamma_x=v-1$.

\begin{lemma}[Block-count identity]\label{lem:blockcount}
Let a $2$-$(v,\{k_1<\cdots<k_r\},1)$ covering have $m_i$ blocks of size $k_i$.  Then
\begin{equation}\label{eq:blockcount-general}
        |\calB|=\frac{\binom v2+|E(\ex)|}{\binom{k_r}{2}}
        +\sum_{i=1}^{r-1}m_i\left(1-\frac{\binom{k_i}{2}}{\binom{k_r}{2}}\right).
\end{equation}
In particular, for $\{3,4,5\}$ coverings,
\begin{equation}\label{eq:blockcount}
        c:=\alpha+\beta+\gamma
        =\frac{\binom v2+|E(\ex)|}{10}+\frac{7\alpha+4\beta}{10}.
\end{equation}
\end{lemma}

\begin{proof}
Count covered pairs with multiplicity:
\[
        \sum_{i=1}^r m_i\binom{k_i}{2}=\binom v2+|E(\ex)|.
\]
Solving for $m_r$ and adding $m_1+\cdots+m_{r-1}$ gives \eqref{eq:blockcount-general}.  The special case follows by substituting $\binom32=3$, $\binom42=6$, and $\binom52=10$.
\end{proof}

Thus, after minimum excess has been fixed, the secondary problem for $\{3,4,5\}$ is reduced to minimizing the linear form $7\alpha+4\beta$.  In the decomposition case,
\begin{equation}\label{eq:delta}
        \delta:=10c-\binom v2=7\alpha+4\beta.
\end{equation}

We use standard existence theorems for BIBDs, RBIBDs, PBDs, and GDDs.  In particular, BIBD$(v,3,1)$, BIBD$(v,4,1)$, and BIBD$(v,5,1)$ exist precisely in the usual congruence classes $v\equiv 1,3\pmod6$, $v\equiv 1,4\pmod {12}$, and $v\equiv 1,5\pmod {20}$, respectively \cite{kirkman-1847,hanani-1975}.  We also use the existence theorem for RBIBD$(v,3,1)$ when $v\equiv 3\pmod6$ \cite{RayChaudhuri1971,Lu1990}, and standard PBD/GDD existence results from \cite{handbook-designs,brouwer-1977,Colbourn1992,rees-1989}.

\section{The two-size problem: 3- and 4-cliques}\label{sec:k34}

We record the complete solution for coverings by triangles and 4-cliques in a compact form.  This case is the natural starting point for mixed-clique coverings: once the excess is fixed, the secondary objective becomes a one-parameter problem, and the local congruence arguments developed here reappear in the three-size setting.

Let $\alpha$ and $\beta$ be the numbers of $K_3$'s and $K_4$'s in a $2$-$(v,\{3,4\},1)$ covering, and let $\xi$ be the size of the excess.  Counting edges gives
\begin{equation}\label{eq:k34-count}
        3\alpha+6\beta=\binom v2+\xi,
\end{equation}
so
\begin{equation}\label{eq:k34-objective}
        \alpha+\beta=\frac{1}{6}\left(\binom v2+\xi+3\alpha\right).
\end{equation}
Thus, among coverings with minimum excess, minimizing the number of blocks is equivalent to minimizing the number of triangles.

\begin{lemma}\label{lem:k34-congruence}
For every $2$-$(v,\{3,4\},1)$ covering,
\[
        \xi\equiv \frac{v(1-v)}{2}\pmod3.
\]
In particular, $\xi\equiv2\pmod3$ when $v\equiv2\pmod3$, and $\xi\equiv0\pmod3$ otherwise.
\end{lemma}

\begin{proof}
This follows immediately from \eqref{eq:k34-count} modulo $3$.
\end{proof}

\begin{theorem}\label{thm:k34-excess}
For $v\ge3$,
\[
\xi_{3,4}(v)=
\begin{cases}
0, & v\equiv0,1\pmod3\text{ and }v\ne6,\\
2, & v\equiv2\pmod3,\\
3, & v=6.
\end{cases}
\]
\end{theorem}

\begin{proof}[Proof sketch]
The congruence lower bound is \cref{lem:k34-congruence}.  Decompositions in the zero-excess classes are obtained from BIBDs and standard GDD constructions, with the small orders supplied by direct designs.  If $v\equiv2\pmod3$, the congruence forces at least two repeated edges; the matching upper bounds follow by truncating suitable GDDs and then covering the remaining pair, or by filling the small $K_5$ and $K_8$ cases.  Finally, $K_6$ has no decomposition into triangles and 4-cliques, while three 4-cliques give excess three.
\end{proof}

\begin{theorem}\label{thm:k34-optimum}
The minimum number of blocks in a $2$-$(v,\{3,4\},1)$ covering with minimum excess is given in \cref{tab:k34}.  In particular, the three separately treated orders satisfy
\[
\Cxi(17,\{3,4\},2)=29,\qquad
\Cxi(18,\{3,4\},2)=33,\qquad
\Cxi(19,\{3,4\},2)=35,
\]
with optimum parameter sets $(\alpha,\beta)=(12,17)$, $(15,18)$, and $(13,22)$, respectively \cite{Kovar2025}.  The columns $\alpha_{\min}$ and $\beta_{\max}$ give one optimum parameter set.
\end{theorem}

\begin{center}
\scriptsize
\setlength{\tabcolsep}{2pt}
\begin{longtable}{c c c c c l}
\caption{Optimal $2$-$(v,\{3,4\},1)$ coverings with minimum excess.}\label{tab:k34}\\
\toprule
$v$ & $\xi_{3,4}(v)$ & $\alpha_{\min}$ & $\beta_{\max}$ & $\Cxi(v,\{3,4\},2)$ & source of construction\\
\midrule
\endfirsthead
\toprule
$v$ & $\xi_{3,4}(v)$ & $\alpha_{\min}$ & $\beta_{\max}$ & $\Cxi(v,\{3,4\},2)$ & source of construction\\
\midrule
\endhead
$1,4\pmod {12}$ & $0$ & $0$ & $\dfrac{v^2-v}{12}$ & $\dfrac{v^2-v}{12}$ & BIBD$(v,4,1)$\\[0.8ex]
$7,10\pmod {12}$, $v\notin\{10,19\}$ & $0$ & $7$ & $\dfrac{v^2-v-42}{12}$ & $\dfrac{v^2-v+42}{12}$ & PBD$(v,\{4,7^*\},1)$\\[0.8ex]
$0,3\pmod {12}$ & $0$ & $v/3$ & $\dfrac{v^2-3v}{12}$ & $\dfrac{v^2+v}{12}$ & $4$-GDD of type $3^{v/3}$\\[0.8ex]
$6,9\pmod {12}$ & $0$ & $\dfrac{v+3}{3}$ & $\dfrac{v^2-3v-6}{12}$ & $\dfrac{v^2+v+6}{12}$ & truncated $2$-$(v+1,4,1)$ covering, $v\ge21$\\[0.8ex]
$5,8\pmod {12}$, $v\notin\{8,17\}$ & $2$ & $\dfrac{2v-4}{3}$ & $\dfrac{v^2-5v+12}{12}$ & $\dfrac{v^2+3v-4}{12}$ & truncated $4$-GDD\\[0.8ex]
$2,11\pmod {12}$ & $2$ & $\dfrac{2v-1}{3}$ & $\dfrac{v^2-5v+6}{12}$ & $\dfrac{v^2+3v+2}{12}$ & truncated $4$-GDD\\[0.8ex]
$6$ & $3$ & $0$ & $3$ & $3$ & direct\\
$8$ & $2$ & $4$ & $3$ & $7$ & direct\\
$9$ & $0$ & $12$ & $0$ & $12$ & BIBD$(9,3,1)$\\
$10$ & $0$ & $9$ & $3$ & $12$ & direct\\
$17$ & $2$ & $12$ & $17$ & $29$ & separate analysis \cite{Kovar2025}\\
$18$ & $0$ & $15$ & $18$ & $33$ & separate analysis \cite{Kovar2025}\\
$19$ & $0$ & $13$ & $22$ & $35$ & separate analysis \cite{Kovar2025}\\
\bottomrule
\end{longtable}
\end{center}

\begin{proof}[Proof sketch of \cref{thm:k34-optimum}]
The constructions in \cref{tab:k34} give the upper bounds.  For the lower bounds, \eqref{eq:k34-objective} reduces the problem to minimizing $\alpha$.  The local equation
\[
        2\alpha_x+3\beta_x=v-1+d_\Xi(x)
\]
fixes $\alpha_x$ modulo $3$ once the excess type is known.  In the zero-excess classes this gives the lower bounds $\alpha\ge0$, $\alpha\ge7$, or $\alpha\ge v/3$ according to the residue class.  In the classes with two repeated edges, the excess multigraph is one of $P_3$, a double edge, or $2K_2$; for the orders covered by the generic construction, checking these three possibilities gives $\alpha\ge(2v-4)/3$, with the stated ceiling correction in the classes $2,11\pmod {12}$.  The orders $K_{17}$, $K_{18}$, and $K_{19}$ require separate arguments; the full proofs are given in \cite{Kovar2025} and are not repeated here.
\end{proof}

\section{Minimum excess for 3-, 4-, and 5-cliques}

\begin{theorem}\label{thm:minexcess}
For every $v\ge 3$,
\[
\xi_{3,4,5}(v)=
\begin{cases}
0, & v\notin\{6,8\},\\
2, & v=8,\\
3, & v=6.
\end{cases}
\]
\end{theorem}

\begin{proof}
A $\{K_3,K_4\}$ decomposition of $K_v$ exists for every $v\equiv 0,1\pmod3$ except $v=6$, by standard BIBD and GDD constructions and small direct designs.  Hence $\xi_{3,4,5}(v)=0$ for those orders.

If $v\equiv5\pmod6$, then $v=5$ is trivial, and for $v\ge11$ a PBD$(v,\{3,5^*\},1)$ gives a decomposition into triangles and one quintuple.  If $v\equiv2\pmod6$ and $v\ge20$, take an RBIBD$(v-5,3,1)$; add five new points as a $5$-block and attach them to five parallel classes.  This gives a PBD$(v,\{3,4,5^*\},1)$ and hence a decomposition.  The order $14$ is obtained by the same construction from an RBIBD$(9,3,1)$.

It remains to consider $6$ and $8$.  For $v=6$, three $K_4$'s give excess $3$.  If a minimum-excess covering used a $K_5$ and had excess at most $3$, then $15+\xi=3\alpha+6\beta+10$, so $\xi\equiv1\pmod3$, hence $\xi=1$.  This leaves $(\alpha,\beta)=(2,0)$ or $(0,1)$, but the residual multigraph $K_6+e-K_5$ cannot be decomposed into triangles and 4-cliques.  Thus $\xi_{3,4,5}(6)=3$.

For $v=8$, a $\{K_3,K_4\}$ covering with two repeated edges gives $\xi_{3,4,5}(8)\le2$.  If $\xi=0$, then $3\alpha+6\beta+10\gamma=28$, forcing $\gamma=1$.  But $K_8-K_5=K_3\cup K_{3,5}$ has all degrees odd, so every vertex would have to occur in a $K_4$; this needs at least two $K_4$'s, whereas only one can be packed in $K_3\cup K_{3,5}$.  If $\xi=1$, then $3\alpha+6\beta+10\gamma=29$, so $\gamma=2$.  The two $K_5$'s share one edge, leaving six vertices of degree three in the residual graph; the residual would have to be decomposed only into $K_4$'s, impossible.  Hence $\xi_{3,4,5}(8)=2$.
\end{proof}

\section{General lower bounds for the secondary optimum}

For $v\notin\{6,8\}$, minimum excess is zero and the problem becomes a decomposition problem.

\begin{theorem}[Modular lower bound]\label{thm:modlb}
Let $v\notin\{6,8\}$.  Then
\[
\Cxi(v,\{3,4,5\},2)\ge
\begin{cases}
\left\lceil\dfrac{v^2-v}{20}\right\rceil, & v\equiv1\pmod4,\\[1ex]
\left\lceil\dfrac{v^2+v}{20}\right\rceil, & v\equiv0\pmod4,\\[1ex]
\left\lceil\dfrac{v^2+3v+2}{20}\right\rceil, & v\equiv3\pmod4,\\[1ex]
\left\lceil\dfrac{v^2+5v+2}{20}\right\rceil, & v\equiv2\pmod4.
\end{cases}
\]
\end{theorem}

\begin{proof}
The first case follows from \eqref{eq:blockcount}.  If $v\equiv0\pmod4$, then \eqref{eq:local} gives $\beta_x$ odd for every $x$, so $\beta\ge v/4$.  Substitution into \eqref{eq:blockcount} gives $c\ge(v^2+v)/20$.

If $v\equiv3\pmod4$, then $\beta_x$ is even.  Let $V_1$ be the vertices with $\beta_x\equiv0\pmod4$ and $V_2$ those with $\beta_x\equiv2\pmod4$.  Then vertices in $V_1$ have $\alpha_x$ odd, so $\alpha\ge |V_1|/3$, while $\beta\ge |V_2|/2$.  Hence $3\alpha+2\beta\ge v$.  Also $\binom v2$ is odd, so $\alpha$ is odd and $|V_1|\ge3$.  Therefore \eqref{eq:blockcount} gives $c\ge(v^2+3v+2)/20$.

If $v\equiv2\pmod4$, then $\beta_x$ is odd.  Let $V_1$ be the vertices with $\beta_x\equiv1\pmod4$ and $V_2$ those with $\beta_x\equiv3\pmod4$.  Then $\alpha\ge |V_1|/3$ and $\beta\ge(|V_1|+3|V_2|)/4$, so $6\alpha+4\beta\ge3v$.  Again $\alpha$ is odd and $|V_1|\ge3$, whence $c\ge(v^2+5v+2)/20$.
\end{proof}

\section{Exact residue classes via construction templates}

\begin{construction}[Truncating a $5$-BIBD]\label{con:bibd}
Let a BIBD$(v+s,5,1)$ exist with $s\in\{0,1,2,3\}$.  Deleting $s$ points gives the following decompositions of $K_v$:
\begin{center}
\small
\begin{tabular}{c c c c c}
\toprule
$s$ & $v\pmod {20}$ & $\alpha$ & $\beta$ & $\gamma$\\
\midrule
$0$ & $1,5$ & $0$ & $0$ & $(v^2-v)/20$\\
$1$ & $0,4$ & $0$ & $v/4$ & $(v^2-4v)/20$\\
$2$ & $3,19$ & $1$ & $(v-3)/2$ & $(v^2-7v+12)/20$\\
$3$ & $2,18$ & $3$ & $(3v-18)/4$ & $(v^2-10v+36)/20$\\
\bottomrule
\end{tabular}
\end{center}
\end{construction}

\begin{construction}[Hole filling]\label{con:hole}
If there is a PBD$(v,\{5,h^*\},1)$ and $K_h$ decomposes into $a$ triangles, $b$ 4-cliques and $g$ quintuples, then filling the hole gives a decomposition of $K_v$ with
\[
        \alpha=a,
        \qquad \beta=b,
        \qquad \gamma=g+\frac{\binom v2-\binom h2}{10}.
\]
We use this with $h=13$, where $K_{13}$ decomposes into $13$ copies of $K_4$, and with $h=17$, where $K_{17}$ decomposes into $16$ copies of $K_4$ and $4$ copies of $K_5$.
\end{construction}

\begin{theorem}\label{thm:exact}
For a minimum-excess $2$-$(v,\{3,4,5\},1)$ covering,
\[
\Cxi(v,\{3,4,5\},2)=
\begin{cases}
\dfrac{v^2-v}{20}, & v\equiv1,5\pmod {20},\\[1ex]
\dfrac{v^2+v}{20}, & v\equiv0,4\pmod {20},\\[1ex]
\dfrac{v^2+3v+2}{20}, & v\equiv3,19\pmod {20},\\[1ex]
\dfrac{v^2+5v+6}{20}, & v\equiv2,18\pmod {20}.
\end{cases}
\]
\end{theorem}

\begin{proof}
The lower bounds are \cref{thm:modlb}, with ceilings evaluated in the stated residue classes.  The matching decompositions are supplied by \cref{con:bibd} and the existence theorem for BIBD$(n,5,1)$.
\end{proof}

\section{Elimination tools for the remaining residue classes}

\begin{lemma}[Master elimination tests]\label{lem:elim}
Let $K_v$ have a $\{K_3,K_4,K_5\}$ decomposition with $c=\alpha+\beta+\gamma$, and put $\delta=10c-\binom v2$.  Then:
\begin{enumerate}[label=(\roman*)]
\item $7\alpha+4\beta=\delta$;
\item
\[
\left\lceil\frac{v^2-v-12c}{8}\right\rceil\le\gamma\le
\left\lfloor\frac{v^2-v-6c}{14}\right\rfloor;
\]
\item $\gamma\equiv\binom v2\pmod3$;
\item if $v$ is odd and $\beta>0$, then $\beta\ge5$;
\item if $v\equiv1\pmod4$, $\alpha=0$, and $\beta>0$, then $\beta\ge13$.
\end{enumerate}
\end{lemma}

\begin{proof}
Part (i) is \eqref{eq:delta}.  For (ii), use
\[
3(c-\gamma)\le3\alpha+6\beta=\binom v2-10\gamma\le6(c-\gamma).
\]
Part (iii) follows from $\binom v2-10\gamma=3\alpha+6\beta\equiv0\pmod3$.  For (iv), $v$ odd implies $\beta_x$ even at each vertex.  Thus any vertex incident with a 4-clique is incident with at least two 4-cliques.  If $n$ vertices are incident with 4-cliques, then $\beta\ge n/2$.  The packing values for $D(n,4,2)$ show that $D(n,4,2)\ge n/2$ first occurs at $n=10$, so $\beta\ge5$ \cite{handbook-designs}.  Part (v) is the same argument with $\beta_x\equiv0\pmod4$, giving first possible support size $13$.
\end{proof}

\begin{lemma}\label{lem:forbidden13}
If $v\equiv13\pmod {20}$, then no decomposition of $K_v$ has $(\alpha,\beta)=(2,7)$.
\end{lemma}

\begin{proof}
Let $U$ be the set of vertices incident with a triangle or 4-clique.  For $x\in U$,
$2\alpha_x+3\beta_x\equiv0\pmod4$, hence $\beta_x$ is even.  Split $U=V_1\cup V_2$ according as $\beta_x\equiv0$ or $2\pmod4$.  With $\alpha=2$ and $\beta=7$, vertices in $V_2$ have $\alpha_x=1$ and $\beta_x\in\{2,6\}$, while vertices in $V_1$ have $\alpha_x\in\{0,2\}$ and $\beta_x\in\{0,4\}$.

No vertex in $V_2$ can have $\beta_x=6$: the six 4-cliques through such a vertex create $18$ further incidences that would each have to be supported by another 4-clique, but only one 4-clique remains.  Hence each vertex of $V_2$ has $\beta_x=2$.  If the two triangles meet, then $|V_2|=4$ and five further vertices have $\beta_x=4$; if the triangles are disjoint, then $|V_2|=6$ and four further vertices have $\beta_x=4$.  In both cases all triangle and 4-clique edges are supported on at most ten vertices, contradicting
\[
        2\binom32+7\binom42=48>\binom{10}{2}=45.
\]
\end{proof}

\begin{lemma}\label{lem:forbidden917}
If $v\equiv9,17\pmod {20}$, then no decomposition of $K_v$ has $(\alpha,\beta)=(2,5)$ or $(2,10)$.
\end{lemma}

\begin{proof}
For $(2,5)$, the same parity split as in \cref{lem:forbidden13} applies.  A vertex with $\beta_x=4$ would force the twelve other incidences in its four 4-cliques to be supported by the single remaining 4-clique, impossible.  Thus every vertex incident with a 4-clique has $\beta_x=2$.  Depending on whether the two triangles meet or not, this gives $\beta=2$ or $3$, not $5$.

For $(2,10)$, let $S=\{x:\beta_x>0\}$.  Since each 4-clique through $x$ uses three distinct neighbours in $S$,
\[
        40=\sum_{x\in S}\beta_x\le |S|\left\lfloor\frac{|S|-1}{3}\right\rfloor,
\]
so $|S|\ge13$.  If the two triangles meet, four vertices have $\alpha_x=1$ and require $\beta_x\ge2$, while all other vertices in $S$ have $\beta_x\ge4$; hence $40\ge8+4(|S|-4)$, giving $|S|\le12$, contradiction.  If the triangles are disjoint, the same argument gives $|S|\le13$, so $|S|=13$.  Then seven vertices of $S$ have $\beta_x=4$ and six have $\beta_x=2$.  The edges from the seven high-degree vertices to $S$ number $\binom72+7\cdot6=63$, but ten 4-cliques contain only $60$ edges, contradiction.
\end{proof}

\section{Bounds in the remaining residue classes}

\begin{theorem}\label{thm:bounds}
For minimum-excess $2$-$(v,\{3,4,5\},1)$ coverings, the following hold.
\begin{center}
\begin{tabular}{c c}
\toprule
Residue class & Value or bounds for $\Cxi(v,\{3,4,5\},2)$\\
\midrule
$v\equiv9,17$, $v\ge89$ or $v=69$ & $\dfrac{v^2-v+128}{20}$\\[1.2ex]
$v\equiv13$, $v\ge53$ & $\dfrac{v^2-v+84}{20}\le \Cxi\le\dfrac{v^2-v+104}{20}$\\[1.2ex]
$v\equiv12$, $v\ge52$ & $\dfrac{v^2+v+24}{20}\le \Cxi\le\dfrac{v^2+v+104}{20}$\\[1.2ex]
$v\equiv8,16$, $v\ge88$ & $\dfrac{v^2+v+8}{20}\le \Cxi\le\dfrac{v^2+v+128}{20}$\\[1.2ex]
$v\equiv11$, $v\ge51$ & $\dfrac{v^2+3v+6}{20}\le \Cxi\le\dfrac{v^2+3v+106}{20}$\\[1.2ex]
$v\equiv7,15$, $v\ge87$ & $\dfrac{v^2+3v+10}{20}\le \Cxi\le\dfrac{v^2+3v+130}{20}$\\[1.2ex]
$v\equiv10$, $v\ge50$ & $\dfrac{v^2+5v+10}{20}\le \Cxi\le\dfrac{v^2+5v+110}{20}$\\[1.2ex]
$v\equiv6,14$, $v\ge86$ & $\dfrac{v^2+5v+14}{20}\le \Cxi\le\dfrac{v^2+5v+134}{20}$\\
\bottomrule
\end{tabular}
\end{center}
The lower bounds hold whenever the minimum excess is zero; the restrictions on $v$ are for the upper-bound constructions.
\end{theorem}

\begin{proof}
The upper bounds come from \cref{con:hole} and its truncations.  The following table gives the parameters of the constructions.
\begin{center}
\scriptsize
\setlength{\tabcolsep}{3pt}
\begin{tabular}{c c c c c}
\toprule
$v\pmod {20}$ & hole & trunc. & $(\alpha,\beta,\gamma)$ & $c$\\
\midrule
$13$ & $13$ & $0$ & $(0,13,(v^2-v-156)/20)$ & $(v^2-v+104)/20$\\
$9,17$ & $17$ & $0$ & $(0,16,(v^2-v-192)/20)$ & $(v^2-v+128)/20$\\
$12$ & $13$ & $1$ & $(4,(v+24)/4,(v^2-4v-96)/20)$ & $(v^2+v+104)/20$\\
$8,16$ & $17$ & $1$ & $(0,(v+64)/4,(v^2-4v-192)/20)$ & $(v^2+v+128)/20$\\
$11$ & $13$ & $2$ & $(5,(v+9)/2,(v^2-7v-84)/20)$ & $(v^2+3v+106)/20$\\
$7,15$ & $17$ & $2$ & $(1,(v+29)/2,(v^2-7v-180)/20)$ & $(v^2+3v+130)/20$\\
$10$ & $13$ & $3$ & $(7,(3v+6)/4,(v^2-10v-60)/20)$ & $(v^2+5v+110)/20$\\
$6,14$ & $17$ & $3$ & $(3,(3v+46)/4,(v^2-10v-156)/20)$ & $(v^2+5v+134)/20$\\
\bottomrule
\end{tabular}
\end{center}

For $v\equiv13\pmod {20}$, the possible values of $\delta$ below the upper bound are $2,12,22,32,42,52$.  If $\beta>0$, then \cref{lem:elim} gives $\delta\ge7\alpha+20$.  The values $22$ and $32$ lead to $(\alpha,\beta)=(2,2)$, $(0,8)$, or $(4,1)$, all impossible by \cref{lem:elim}.  If $\beta=0$, then $\delta=7\alpha$, so the first possible value is $42$.  Thus $\delta\ge42$, equivalent to the lower bound $(v^2-v+84)/20$.  The additional configuration $(\alpha,\beta)=(2,7)$ at $\delta=42$ is ruled out by \cref{lem:forbidden13}.

For $v\equiv9,17\pmod {20}$, the possible smaller values are $\delta=4,14,24,34,44,54$.  If $\beta=0$, then $\delta=14$ and $\alpha=2$, but the local congruence $2\alpha_x+4\gamma_x\equiv0\pmod4$ would require each vertex incident with a triangle to lie in an even number of triangles, impossible for two triangles.  Hence $\beta>0$.  The master tests eliminate $\delta=24$ and $44$; $\delta=34$ gives $(2,5)$; and $\delta=54$ gives $(2,10)$ or $(6,3)$, the latter violating $\beta\ge5$.  The pairs $(2,5)$ and $(2,10)$ are excluded by \cref{lem:forbidden917}.  Thus $\delta\ge64$, which matches the construction.

The other lower bounds are \cref{thm:modlb}, except for $v\equiv12\pmod {20}$.  There the modular lower bound gives $(v^2+v+4)/20$.  If equality held, then the proof of \cref{thm:modlb} and \eqref{eq:blockcount} force $\alpha=0$.  Solving
\[
\beta+\gamma=\frac{v^2+v+4}{20},\qquad 6\beta+10\gamma=\binom v2
\]
gives $\beta=(v+2)/4\notin\ZZ$, contradiction.  Hence one more block is needed.
\end{proof}

\section{Small orders}\label{sec:small33}

Two triples are called \emph{independent} if they are vertex-disjoint.

\begin{theorem}\label{thm:small}
For $3\le v\le17$, the values are as follows.
\begin{center}
\begin{tabular}{c c c l}
\toprule
$v$ & $\xi_{3,4,5}(v)$ & $\Cxi(v,\{3,4,5\},2)$ & construction or lower-bound reason\\
\midrule
3 & 0 & 1 & one $K_3$\\
4 & 0 & 1 & one $K_4$\\
5 & 0 & 1 & one $K_5$\\
6 & 3 & 3 & three $K_4$'s\\
7 & 0 & 7 & BIBD$(7,3,1)$\\
8 & 2 & 7 & ILP by excess type; optimum $(4,3,0)$\\
9 & 0 & 12 & BIBD$(9,3,1)$\\
10 & 0 & 12 & minimum $\{K_3,K_4\}$ decomposition\\
11 & 0 & 16 & one $K_5$ and $15$ triangles\\
12 & 0 & 13 & minimum $\{K_3,K_4\}$ decomposition\\
13 & 0 & 13 & BIBD$(13,4,1)$\\
14 & 0 & 19 & $(\alpha,\beta,\gamma)=(9,9,1)$\\
15 & 0 & 20 & minimum $\{K_3,K_4\}$ decomposition\\
16 & 0 & 20 & BIBD$(16,4,1)$\\
17 & 0 & 20 & $16$ 4-cliques and $4$ quintuples\\
\bottomrule
\end{tabular}
\end{center}
\end{theorem}

\begin{proof}
The upper bounds are the displayed constructions.  The lower bounds use \cref{thm:modlb} together with the following cross-edge bound when a quintuple is present.
\end{proof}

\begin{lemma}[Cross-edge bound]\label{lem:cross}
If $K_v$ has a decomposition into $\alpha$ triangles, $\beta$ 4-cliques and $\gamma\ge1$ quintuples, with $c=\alpha+\beta+\gamma$, then
\[
        c\ge\max\left\{
        \left\lceil\frac{-v^2+41v-168}{12}\right\rceil,
        \left\lceil\frac{-v^2+31v-124}{6}\right\rceil
        \right\}.
\]
\end{lemma}

\begin{proof}
Fix one quintuple.  The remaining blocks must cover the $5(v-5)$ edges from it to the other vertices.  A triangle, 4-clique, or another quintuple covers at most $2,3,4$ of these crossing edges, respectively, so
\[
        5(v-5)\le2\alpha+3\beta+4(\gamma-1).
\]
Eliminating $\alpha$ and $\beta$ using $c=\alpha+\beta+\gamma$ and $3\alpha+6\beta+10\gamma=\binom v2$ gives
\[
        c\ge -\frac{v^2}{6}+\frac{31v}{6}+\frac{\gamma}{3}-21.
\]
Now use $\gamma\ge1$ and the lower bound on $\gamma$ from \cref{lem:elim}.
\end{proof}

The first unresolved order in the class $v\equiv13\pmod {20}$ is $v=33$.  The general lower bound in \cref{thm:bounds} gives $\Cxi(33,\{3,4,5\},2)\ge57$.  In a companion preprint \cite{KovarZhang2026K33}, this was improved to
\[
        \Cxi(33,\{3,4,5\},2)\ge59
\]
by combining theoretical reductions with an exact filtering and SAT-based verification.  We do not reproduce that certificate here; we cite it only to indicate the current boundary of the unresolved residue class.

\section{Concluding remarks}

The two-size problem for $K_3$ and $K_4$ provides the basic local-congruence and truncation methods, and its unavoidable excess for $v\equiv2\pmod3$ motivates the introduction of quintuples.  For $\{K_3,K_4,K_5\}$ coverings, the minimum excess is completely determined: it is zero except for $v=6$ and $v=8$.  The block-count identity then organizes the secondary optimization and avoids a separate ad hoc argument for each residue class.

The secondary optimum is exact in the residue classes
\[
0,1,2,3,4,5,9,17,18,19\pmod {20},
\]
with the stated lower bound on $v$ in the classes $9$ and $17$, and it is also determined for all $3\le v\le17$.  For the remaining residue classes the bounds in \cref{thm:bounds} leave finite gaps.  Closing them will likely require sharper structural exclusions for small candidate pairs $(\alpha,\beta)$, or computer-verifiable local analyses of the kind used for the separate $K_{33}$ result in \cite{KovarZhang2026K33}.

\paragraph{Acknowledgements.}
The authors thank Mariusz Meszka for suggestions on applying resolvable designs.

\paragraph{Funding.}
This work is partially supported by Grant of SGS No. SP2025/012 and No. SP2025/049, VSB--Technical University of Ostrava, Czech Republic. Petr Kov\'a\v{r} is co-funded by the financial support of the European Union under the REFRESH--Research Excellence For Region Sustainability and High-tech Industries project number CZ.10.03.01/00/22 003/0000048 via the Operational Programme Just Transition.

\paragraph{Data availability.}
No new computational data are required for the results proved in this paper. 

\bibliographystyle{plainurl}
\bibliography{refs}

\end{document}